\documentclass[12pt]{article}
\makeatletter
\newcommand{\rmnum}[1]{\romannumeral #1}
\newcommand{\Rmnum}[1]{\expandafter\@slowromancap\romannumeral #1@}
\makeatother
\usepackage{}
\usepackage[numbers,sort&compress]{natbib}
\usepackage{color}
\usepackage{threeparttable}
\usepackage[english]{babel}
\usepackage{diagbox}
\usepackage{t1enc}
\usepackage{amsthm,amsmath,amssymb}
\usepackage{mathrsfs}
\usepackage[mathscr]{eucal}
\usepackage[latin1]{inputenc}
\usepackage[english]{babel}
\usepackage{latexsym}
\usepackage{graphicx}
\usepackage[noend]{algpseudocode}
\usepackage{algorithmicx,algorithm}
\usepackage{booktabs}
\usepackage{multirow}
\newtheorem{theorem}{Theorem}
\newtheorem{question}{Question}

\newtheorem{proposition}{Proposition}
\newtheorem{lemma}{Lemma}
\newtheorem{remark}{Remark}

\def \col{\textup{col}}
\def \diag{\textup{diag}}

\def \Res{\textup{Res}}
\def \spec{\textup{spec}}
\def \span{\textup{span}}
\def \T{\textup{T}}
\def \S{\textup{S}}
\def\rank{\textup{rank}}
\def \rrg{\textup{RO}_n(\mathbb{Q})}
\begin{document}
	\title{Generalized spectral characterization of signed bipartite graphs}
	\author{\small  Songlin Guo$^{{\rm a}}$ \quad Wei Wang$^{{\rm b}}$\thanks{Corresponding author. Email address: wangwei.math@gmail.com} \quad Lele Li$^{{\rm a}}$
		\\
		{\footnotesize$^{\rm a}$ School of General Education, Wanjiang University of Technology, Ma'anshan 243031, P. R. China}\\
		{\footnotesize$^{\rm b}$School of Mathematics, Physics and Finance, Anhui Polytechnic University, Wuhu 241000, P. R. China}
	}
	\date{}
	\maketitle
	\begin{abstract}
		Let $\Sigma$ be an $n$-vertex controllable or almost controllable signed bipartite graph, and let $\Delta_\Sigma$ denote the discriminant of its characteristic polynomial $\chi(\Sigma; x)$. We prove that if (\rmnum{1}) the integer $2^{ -\lfloor n/2 \rfloor }\sqrt{\Delta _{\Sigma}}$ is  squarefree, and (\rmnum{2}) the constant term (even $n$) or linear coefficient (odd $n$) of $\chi(\Sigma; x)$ is $\pm 1$, then $\Sigma$ is determined by its generalized spectrum. This result extends a recent theorem of Ji, Wang, and Zhang [Electron. J. Combin. 32 (2025), \#P2.18], which established a similar criterion for signed trees with irreducible characteristic polynomials.
	\end{abstract}

	
	\section{Introduction}
	A signed graph $\Sigma$ is a pair $(G, \sigma)$, where $G$ is a simple  graph and $\sigma$ is a mapping from its edge set $E(G)$ to $\{+1, -1\}$.     The adjacency matrix $A(\Sigma) = (a_{ij})$ of a signed graph $\Sigma=(G,\sigma)$ is defined by:
   \[
    a_{ij} = 
   \begin{cases}
	\sigma(ij) & \text{if  $i$ and $j$ are adjacent;}\\
	0 & \text{otherwise}.
   \end{cases}
   \]
   Two signed graphs $\Sigma_1$  and $\Sigma_2$ are said to be generalized cospectral \cite{ji2025} if their adjacency matrices $A(\Sigma_1)$ and $A(\Sigma_2)$ satisfy:
       \begin{enumerate}
       	\item  $\det(xI-A(\Sigma_1))=\det(xI-A(\Sigma_2))$ and
       	\item  $\det(xI-(J-I-A(\Sigma_1)))=\det(xI-(J-I-A(\Sigma_2)))$,
       \end{enumerate}
    where $I$ is the identity matrix, $J$ is the all-ones matrix, and $(J-I-A(\Sigma_1))$ formally denotes the `complement' of $\Sigma_1$.  
    A signed graph $\Sigma$ is determined by its generalized spectrum (DGS) \cite{ji2025} if every signed graph generalized cospectral with $\Sigma$ is isomorphic to it.
    
    For convenience, we say a signed graph $\Sigma$ is \emph{reducible} (resp. \emph{irreducible}) if its characteristic polynomial $\chi(\Sigma;x)$, which is $\det(xI-A(\Sigma))$, is reducible (resp. irreducible) over $\mathbb{Q}$. We use $\Delta_\Sigma$ to denote the discriminant of  $\chi(\Sigma;x)$. Recently, Ji, Wang and Zhang \cite{ji2025} obtained a simple criterion for a signed tree to be DGS. We state it in a slightly different but essentially equivalent form.
     \begin{theorem}[\cite{ji2025}]\label{main2}
    	Let $\Sigma$ be an $n$-vertex ($n\ge 2$) irreducible signed tree. If  $2^{-n/2}\sqrt{\Delta_{\Sigma}}$ is an odd squarefree integer, then $\Sigma$  is DGS.  
    \end{theorem}
    We note that the order $n$ is necessarily even in Theorem \ref{main2} by the irreducibility assumption. Indeed, if $\Sigma$ is a signed tree (or more generally, a signed bipartite graph) whose order is odd, then it is not difficult to see that $A(\Sigma)$ is singular and hence  $\chi(\Sigma;x)$ has $x$ a factor. Furthermore, since for any signed tree $\Sigma$, the constant term of its characteristic polynomial $\chi(\Sigma;x)$ belongs to $\{0,1,-1\}$, the irreducibility assumption in Theorem \ref{main2} clearly implies that the constant term of $\chi(\Sigma;x)$  is $\pm 1$.
    
    At the end of \cite{ji2025}, Ji et al. proposed the following question for further study.
        \begin{question}[\cite{ji2025}]\label{qqqq}
    	How can Theorem \ref{main2} be generalized to signed bipartite graphs?
    \end{question} 
    In the same paper, Ji et al. realized that some essential difficulties will inevitably appear when we try to generalize Theorem \ref{main2} to signed bipartite graphs. They also reported a `counterexample' indicating that Theorem 7 in that paper, which is the main tool to prove Theorem \ref{main2}, does not hold without the irreducibility  assumption  of $\Sigma$,  even if $\Sigma$ is controllable.  Thus it is not clear how to  generalize Theorem \ref{main2} to reducible signed trees (which include all signed trees with $2k+1$ vertices). 
    
    In this paper, we employ a new approach to give an answer to Question \ref{qqqq}. The main result of this paper is the following theorem.
    
    \begin{theorem}\label{main}
    	Let $\Sigma$ be an $n$-vertex controllable or almost controllable signed bipartite graph. Assume that the coefficient of the constant term ($n$ even) or linear term ($n$ odd) of $\chi(\Sigma;x)$ is $\pm 1$.  If  $2^{-n/2}\sqrt{\Delta_{\Sigma}}$ is squarefree, then $\Sigma$  is DGS. 
    \end{theorem}
  Compared with Theorem \ref{main2}, Theorem \ref{main} has broader applicability and subsumes Theorem 1 as a special case. Note that Theorem \ref{main} no longer requires the oddness condition explicitly, since we will prove that the value $2^{-n/2}\sqrt{\Delta_{\Sigma}}$ must be odd when  squarefree.
    
    \section{Preliminaries}
    We first recall some basic notations. Let $\Sigma$ be an $n$-vertex signed graph. The walk-matrix of $\Sigma$ is defined as:
    \begin{eqnarray*}
    	&&W(\Sigma):= [e,Ae,\ldots,A^{n-1}e], 
    \end{eqnarray*} 
    where $e$ is the all-ones  vector and $A$ is the adjacency matrix of $\Sigma$. We say $\Sigma$ is controllable (resp. almost controllable) if $\rank (W(\Sigma))=n$ (resp. $\rank(W(\Sigma))=n-1$).
    
    \subsection{Totally isotropic subspace in generalized cospectrality}
    
     An orthogonal matrix $Q$ is called regular if it satisfies $Qe=e$. The following theorem  states that for controllable or almost controllable signed graphs, generalized cospectrality can be characterized by a regular orthogonal matrix with rational entries.  The result was usually stated and proved in the setting of unsigned graphs; nevertheless, the original proofs are still valid for signed graphs.    
     \begin{theorem}
    	[\cite{johnson1980JCTB, wang2006EUJC, wang2021EUJC}]\label{connection}
    	Let $\Sigma$ and $\Gamma$ be two signed graphs with $n$ vertices. Then $\Sigma$ and $\Gamma$ are generalized cospectral if and only if there exists a regular orthogonal matrix $Q$ such that
    	\begin{equation}\label{qaq}
    		Q^\T A(\Sigma)Q=A(\Gamma).
    	\end{equation} Moreover,
    	
    	\textup{(\rmnum{1})} if $\Sigma$ is controllable then $Q^\T=W(\Gamma)(W(\Sigma))^{-1}$ and hence $Q$ is unique and rational.
    	
    	\textup{(\rmnum{2})} if $\Sigma$ is almost controllable then  Eq.~\eqref{qaq} has exactly two solutions for $Q$, both of which are rational.
    \end{theorem}
      Let $\rrg$ and $\S_n(\mathbb{Z})$ denote the sets of all $n\times n$ rational regular orthogonal matrices and all $n\times n$ symmetric integer matrices, respectively. For a signed graph $\Sigma$, we define 
    $$\mathcal{Q}(\Sigma)=\{Q\in \rrg\colon\, Q^\T A(\Sigma)Q \in \S_n(\mathbb{Z}) \}.$$ 
 For a matrix  $Q\in O_n(\mathbb{Q})$, its level, denoted by $\ell(Q)$ (or simply write it as $\ell$), is defined as the smallest positive integer $k$ such that $k Q$ is an integer matrix.  Clearly, matrices in $\rrg$ with level 1 are precisely the permutation matrices. The following observation is standard and frequently used to show a graph (or a signed graph) to be DGS.
 \begin{lemma}\label{dgslev}
 	Let $\Sigma$ be a controllable or almost controllable signed graph. If each matrix $Q\in \mathcal{Q}(\Sigma)$ has level 1 then $\Sigma$ is DGS.
 \end{lemma}
 \begin{remark}\normalfont{
 	For almost controllable (signed) graphs,  Lemma \ref{dgslev} is only applicable for those that have a nontrivial automorphism. Indeed, if $\Sigma$ is almost controllable but has no nontrivial automorphism, then $\mathcal{Q}(\Sigma)$ contains a non-permutation matrix and hence the condition of Lemma \ref{dgslev} fails.
}
 \end{remark}

 Let $p$ be a prime. We consider the $n$-dimensional vector space $\mathbb{F}^n_p$ over the finite field $\mathbb{F}_p=\mathbb{Z}/p\mathbb{Z}$, consisting of all column vectors $(x_1, x_2, ..., x_n)^\T$ with components $x_i \in \mathbb{F}_p$. This space is endowed with the standard inner product defined by $\langle u,v \rangle = u^\T v$.  Two vectors $u, v \in \mathbb{F}^n_p$ are called \emph{orthogonal}, denoted by  $u \perp v$, if $u^\T v = 0$. Similarly, two subspaces $U$ and $V$ are \emph{orthogonal}, denoted by $U\perp V$, if $u\perp v$ for any $u\in U$ and  $v\in V$.  A nonzero vector $u\in \mathbb{F}_p^n$ is \emph{isotropic} if $u\perp u$, i.e., $u^\T u=0$.  For a subspace $V$ of $\mathbb{F}_p^n$, the \emph{orthogonal space} of $V$ is
\begin{equation*}
	V^\perp=\{u\in \mathbb{F}_p^n\colon\, v^\T u=0\text{~for every $v\in V$}\}.
\end{equation*}
A subspace $V$ of $\mathbb{F}_p^n$ is \emph{totally isotropic} \cite{babai2022} if $V\subset V^\perp$, i.e., every pair of vectors  in $V$ are orthogonal.

Let $M$ be an $n\times n$ matrix over $\mathbb{F}_p$. We identify $M$ with the linear transformation  $x\mapsto Mx$ for $x\in \mathbb{F}_p^n$.
A subspace $V$ of $\mathbb{F}_p^n$ is $M$-\emph{invariant} if $Mx\in V$ for any $x\in V$.  
%\begin{lemma}[\cite{Primary}]\label{pd}
%	Let $M$ be an $n\times n$ symmetric matrix over $\mathbb{F}_p$. If there exists  a nontrivial totally isotropic subspace $U$, then $\chi(M;x)$ has a multiple factor $\phi(x)$ such that $U\cap \ker \phi(A)\neq 0$.
%\end{lemma}

 \begin{lemma}[\cite{Primary}]\label{pd}
	Let $\chi(M;x)\in \mathbb{F}_p[x]$ be the characteristic polynomial of $M$ and
	$$\chi(M;x)=\phi_1^{r_1}(x)\phi_2^{r_2}(x)\cdots\phi_k^{r_k}(x),$$
	be the standard factorization of $\chi(M;x)$. Let $U$ be an $M$-invariant subspace and denote $V_i=\ker \phi^{r_i}_i(M), i = 1,\ldots, k$. Then\\	
	(\rmnum{1}) if $\phi_i(x)$ is a simple factor (i.e., $r_i=1$), then $U\cap V_i=0$;\\
	(\rmnum{2}) if $U\cap\ker \phi_i^{r_i}(M)$ is nonzero, then  $U \cap\ker \phi_i(M)$ is also nonzero;\\
	(\rmnum{3}) $U=\oplus (U\cap V_i)$, where the summation is taken over all subscripts $i$ satisfying $r_i\ge 2$.\\
\end{lemma}   
Suppose that $Q\in \mathcal{Q}(\Sigma)$ such that $\ell:=\ell(Q)>1$. We define 
\begin{equation*}
	\hat{Q}=\ell\cdot Q\in \mathbb{Z}^{n\times n}.
\end{equation*}
Let $M\in \mathbb{Z}^{n\times n}$ be an integer matrix and $p$ be a fixed prime. We use $\col_p(M)$ and $\ker_p(M)$ to denote the column space and  null space (or kernel), which are subspaces of $\mathbb{F}_p^n$. 

\begin{lemma}[\cite{Primary}]\label{ainv}
	For any prime factor $p$ of $\ell(Q)$, the space $\col_p(\hat{Q})$  is nonzero, totally isotropic and $A$-invariant.
\end{lemma}

The following proposition is  immediate from Lemmas \ref{pd} and \ref{ainv}.
\begin{proposition}\label{dfphi}
For any prime factor $p$ of $\ell(Q)$, the characteristic polynomial $\chi(A;x)$ has a multiple factor $\phi(x)$ such that
$$\col(\hat{Q})\cap \ker \phi(A)\neq 0, \text{~over~}\mathbb{F}_p.$$
\end{proposition}
	\begin{proposition}[\cite{Primary}]\label{strict} 
Let $p$ be an odd prime factor of $\ell(Q)$ and $\phi(x)\in \mathbb{Z}[x]$ satisfy the conclusion of Proposition \ref{dfphi}. Then we have
	\begin{equation*}
		p^{\deg \phi(x)+1}\mid\det \phi(A).
	\end{equation*}
\end{proposition}
\subsection{Resultant and discriminant}
    
For a monic polynomial $f(x)=x^n+a_{n-1}x^{n-1}+\ldots+a_{1}x+a_0 \in Z[x]$, the discriminant of $f(x)$ is defined as:
	\begin{eqnarray*}
		&&\Delta(f) = \prod _{1\leq i< j\leq n}\left( \alpha _{j}-\alpha _{i}\right) ^{2},
	\end{eqnarray*} 
	where $\alpha_1,\ldots,\alpha_n$ are the roots of $f(x)$ in $\mathbb{C}$. The resultant of $f$ and its derivative $f'$, denoted by $\Res(f,f')$,  is the determinant of the $(2n-1)\times(2n-1)$ Sylvester matrix
	\begin{equation}
		S(f,f')=\begin{pmatrix} 
			1 & a_{n-1} &\ldots &\ldots&\ldots&a_0 & & & &         \\
			& 1 & a_{n-1} &\ldots &\ldots&\ldots&a_0 & & &         \\ 
			&&\ldots &\ldots&\ldots& \ldots&                   \\ 
			&&&1 & a_{n-1} &\ldots &\ldots&\ldots&a_0              \\
			n & (n-1)a_{n-1} &\ldots &\ldots&a_{1} & & & & &               \\
			&n & (n-1)a_{n-1} &\ldots &\ldots&a_{1} & & & &                \\ 
			&&\ldots &\ldots&\ldots& \ldots&                   \\
			& & & & n & (n-1)a_{n-1} &\ldots &\ldots&a_{1}  
		\end{pmatrix}.              
	           \end{equation}
	      It is known that $\Delta(f)=\pm \Res(f,f')$ and hence $\Delta(f)$ is an integer.
  \begin{lemma}[\cite{Lang}]\label{mfbasic}
Let $p$ be a prime and $f(x) \in \mathbb{Z}[x]$ be a monic polynomial. Then $f(x)$ has a multiple factor over $\mathbb{F}_p$ if and only if 	$p| \Delta(f)$.
\end{lemma} 
  
    Let $M\in \S_n(\mathbb{Z})$, we use $\Delta_M$ to denote $\Delta(\chi(M;x))$, the discriminant of $\chi(M;x)$. For two polynomial $f(x),g(x)\in \mathbb{Z}[x]$ and an integer $q$, we denote $f(x)\equiv g(x)\pmod{q}$ if all corresponding coefficients of $f$ and $g$ are congruent modulo $q$. The following three lemmas were obtained by Wang and Yu \cite[Theorem 3.3, Lemma 4.3, and Lemma 4.4]{wangyu2016}):
    \begin{lemma}[\cite{wangyu2016}]\label{dd}
    	Let $M\in \S_n(\mathbb{Z})$ and $Q$ be a rational orthogonal matrix such that $Q^\T MQ\in \S_n(\mathbb{Z})$.  Then any prime factor of $\ell(Q)$ is a factor of $\Delta_M$. 
    \end{lemma}
    \begin{lemma}[\cite{wangyu2016}]\label{d}
    Let  $M\in \S_n(\mathbb{Z})$ and $p$ be any odd prime. If $p\mid\Delta _{M}$ but $p^2\nmid\Delta _{M}$, then there exists an integer $\lambda_0$ and a polynomial $\varphi(x)$ with integer coefficients such that  $\chi(M;x)\equiv(x-\lambda_0)^2\varphi(x)\pmod{p}$, where $\varphi(x)$ is squarefree over $\mathbb{Z}_p$ and $ \varphi(\lambda_0)\not\equiv 0\pmod{p}$.
    \end{lemma}
    \begin{lemma}[\cite{wangyu2016}]\label{d2}
	Let  $M\in \S_n(\mathbb{Z})$  and $p$ be any odd prime. Suppose that $\chi(M;x)\equiv(x-\lambda_0)^2\varphi(x)\pmod{p}$ for some $\lambda_0\in\mathbb{Z}$ and $\varphi(x)\in \mathbb{Z}[x]$, where  $\varphi(x)$ is squarefree over $\mathbb{F}_p$ and $\varphi(\lambda_0)\not\equiv 0\pmod{p}$. Then the equation
	\begin{eqnarray*}
		\chi(M;x)u(x)\equiv \chi'(M;x) v(x)\pmod{p^2}
	\end{eqnarray*}
	has a solution $(u(x),v(x))\in (\mathbb{Z}[x])^2$ with:\\
	\noindent(\rmnum{1}) $u(x),v(x)\not\equiv0\pmod{p}$;\\
	\noindent(\rmnum{2}) $\deg(u(x)) <n-1= \deg(\chi'(M;x))$;\\
	\noindent(\rmnum{3}) $\deg(v(x)) < n=\deg(\chi(M;x))$,\\
	if and only if $p^2\mid\det(M-\lambda_0I)$.
    \end{lemma}
       \begin{remark}\label{nid}\normalfont{
    		Under the assumption of this lemma, the truth of $p^2\mid \det (M-\lambda_0 I)$ does not depend on the choice of $\lambda_0$ in its residue class, i.e., for any $\lambda_1$ and $\lambda_2$ such that $\lambda_1\equiv \lambda_2\equiv \lambda_0\pmod{p}$, 
    		$p^2\mid \det (M-\lambda_1 I)$ if and only if  	$p^2\mid \det (M-\lambda_2 I)$. This can be easily seen from the conclusion of Lemma \ref{d2} since the existence of a solution $(u,v)$ satisfying the given conditions clearly does not depend on the choice of $\lambda_0$ in its residue class.}
    \end{remark}    
    
    Let $M$ be a nonsingular $n\times n$ integer matrix. It is well-known that there exist two unimodular matrices $U$ and $V$ such that $UMV$ is a diagonal  matrix $S=\diag{(d_1,d_2,\ldots,d_n)}$, where the diagonal entries $d_1,\ldots,d_n$ are positive integers with $d_i\mid d_{i+1}$ for $i=1,2\ldots,n$. The matrix $S$ is unique and is called the Smith normal form of $M$; the diagonal entries are called the invariant factor of $M$.  We summarize some  basic properties of a matrix using its invariant factors.
 \begin{lemma}\label{snf}
	Let $p$ be any prime. For an integral matrix $M$ with invariant factors $d_1,d_2,\ldots,d_n$. We have\\
	\noindent(\rmnum{1}) $p^{n-\rank_p(M)}\mid \det M$;\\
	\noindent(\rmnum{2}) $\det(M)=\pm d_1d_2\ldots d_n$;\\
	\noindent(\rmnum{3}) $\rank_pM=\max\{i:p\nmid d_i\}$;\\
	\noindent(\rmnum{4}) $Mx\equiv0$ $\pmod {p^2}$ has a solution $x\not \equiv  0$ $\pmod {p}$ if and only if $p^2\mid d_n$.	
\end{lemma}

\begin{lemma}\label{degker}
	Let $p$ be a prime and $f(x)=x^n+a_{n-1}x^{n-1}+\cdots+a_1x+a_0$ be a monic polynomial over $\mathbb{F}_p$. Then  $\deg \gcd(f,f')=\dim \ker S(f,f')=(2n-1)-\rank_pS(f,f')$.
\end{lemma}
\begin{proof}
	Let $$u(x)=u_0x^{n-2}+u_1x^{n-3}+\cdots+u_{0},\text{~and~} v(x)=v_{n-1}x^{n-1}+v_{n-2}x^{n-2}+\cdots+v_0.$$
	We consider the equation 
	\begin{equation}\label{euv}
		f(x)u(x)-f'(x)v(x)=0 \text{~with variables~} (u,v).
	\end{equation}  Let $g(x)=\frac{f(x)}{d(x)}$ and $h(x)=\frac{f'(x)}{d(x)}$ where $d(x)=\gcd(f(x),f'(x))$. 
	Then the equation	$f(x)u(x)-f'(x)v(x)=0$ reduces to 
	\begin{equation}\label{gh}
		g(x)u(x)=h(x)v(x).
	\end{equation}
	Noting that $g(x)$ and $h(x)$ are coprime, the solution $(u,v)$ of Eq. \eqref{gh} has the form
	$(u,v)=(h(x)r(x),g(x)r(x))$ for some $r(x)\in \mathbb{F}_p[x]$. Let $k=\deg d(x)$. To satisfy the restrictions of degrees of $u(x)$ and $v(x)$, we need (and only need) $\deg r(x)\le k-1$. This means that the solution space of Eq.~\eqref{gh} (or equivalently Eq.~\eqref{euv}) has dimension $k$. 
	
	Write $\eta=(u_{n-2},u_{n-3},\ldots,u_0,v_{n-1},v_{n-2},\ldots,v_0)$ and $S=S(f,f')$.	Then Eq.~\eqref{euv} is equivalent to $
	S^\T\eta =0$. This means that the solution subspace of Eq.~\eqref{euv} is isomorphic to $\ker S(f,f')$. It follows that $\dim \ker S(f,f')=k=\deg\gcd(f,f')$. This completes the proof.
\end{proof}

\begin{proposition}\label{not2}
	Let $f(x)$ be a monic polynomial with integer coefficients. Then $\Delta(f)\not\equiv 2\pmod{4}$.
\end{proposition}
\begin{proof}
	Let $f(x)=\phi^{r_1}_1(x)\phi^{r_2}_2(x)\cdots\phi^{r_k}_k(x)$ be the standard factorization of $f(x)$ over $\mathbb{F}_2$. We claim that, over $\mathbb{F}_2$, 
	\begin{equation*}
		\gcd(f,f')=\phi^{s_1}_1(x)\phi^{s_2}_2(x)\cdots\phi^{s_k}_k(x),	\end{equation*}
where 
\begin{equation*}
	s_i=\begin{cases} r_i&r_i \text{~even;}\\
	r_i-1 &r_i \text{~odd}.
\end{cases}
\end{equation*}	
Indeed, if some $r_i$, say $r_1$ is odd, then we have $r_1\equiv 1\pmod 2$ and hence
$$f'(x)=\phi_1^{r_1-1}(x)\left(\prod_{j=2}^{k}\phi_j^{r_j}(x)\right)+\phi_1^{r_1}(x)\left(\prod_{j=2}^{k}\phi_j^{r_j}(x)\right)' \text{~over~} \mathbb{F}_2,$$
which implies that the multiplicity of the factor $\phi_1(x)$ in $f'(x)$ is exactly $r_1-1$. Thus, in this case, the multiplicity of $\phi_1(x)$ in $\gcd(f,f')$ is $r_1-1$.   However, if $r_1$ is even, i.e., $r_1\equiv 0\pmod{2}$, then over $\mathbb{F}_2$, we have $(\phi_1^{r_1}(x))'=0$ and hence $$f'(x)=\phi_1^{r_1}(x)\left(\prod_{j=2}^{k}\phi_j^{r_j}(x)\right)' \text{~over~} \mathbb{F}_2,$$
which implies that the multiplicity of $\phi_1(x)$ in $\gcd(f,f')$ is $r_1$. This proves the claim. It follows that the degree of $\gcd(f,f')$ must be even as each multiplicity $s_i$ is even. 

Let $S\in\mathbb{Z}^{(2n-1)\times (2n-1)}$ be the Sylvester matrix of $f$ and $f'$. Suppose to the contrary that $\Delta(f)\equiv 2\pmod{4}$, or equivalently, $\det S\equiv 2\pmod{4}$. Then, by Lemma  (\rmnum{2}), we see that $S$ has exactly one invariant factor that is even, that is, $\rank_2 S=2n-2$ . But this implies $\deg \gcd(f,f')=1$ by Lemma \ref{degker}, contradicting the established fact that the degree of $ \gcd(f,f')$ is  always even. This completes the proof of Proposition \ref{not2}.
\end{proof}

 %   \begin{lemma}[\cite{wangyu2016}, \cite{wang2013ElJC}]\label{snf}
 %   	Let $p$ be any prime. For an integral matrix $M$ with invariant factors $d_1,d_2,\ldots,d_n$, the equation $Mx\equiv0$ $\pmod{p^2}$ has a solution $x\not \equiv  0$ $\pmod {p}$ if and only if $p^2\mid d_n$.	
  %  \end{lemma}

	\section{Proof of Theorem \ref{main}}
For convenience, we define
\begin{equation*}
	\delta=\delta(n)=\lceil n/2 \rceil -\lfloor n/2 \rfloor=\begin{cases}
		0& \text{$n$ even;}\\
		1& \text{$n$ odd.}
	\end{cases}
\end{equation*}
Let $\Sigma$ be an $n$-vertex controllable or almost controllable signed graph such that 
$c_\delta=\pm 1$, where $c_\delta$ is the coefficient of the term $x^\delta$ in $\chi(\Sigma;x)$. As $\Sigma$ is bipartite, we may write the adjacency matrix $A=A(\Sigma)$ in the form
\begin{equation*}
 A=\begin{pmatrix} 
	0   & B \\
	B^\T  & 0 
\end{pmatrix},
\end{equation*} 
	 where $B$ is an $s\times (n-s)$ matrix with $s\le \lfloor n/2\rfloor$. We claim that the equality must hold. Indeed, if $s<\lfloor n/2\rfloor$ then we have 
	 $\rank(A)=2\rank(B)\le 2(\lfloor n/2\rfloor -1)\le n-2$, which implies that $0$ is a multiple eigenvalue of $A$. This contradicts the requirement that $c_\delta=\pm 1$. 
	 
	 \begin{lemma}\label{chiab}
	 	$\chi(A;x)=x^\delta \chi(BB^\T;x^2)=x^{-\delta}\chi(B^\T B;x^2)$.
	 \end{lemma}
	 \begin{proof}
	 	From the identity 
	 	\begin{equation*}
	 		\begin{pmatrix}
	 			xI_{\lfloor n/2\rfloor} &-B\\
	 			-B^\T& xI_{\lceil n/2\rceil}
	 		\end{pmatrix}\begin{pmatrix}I_{\lfloor n/2\rfloor}&0\\\frac{1}{x}B^\T &I_{\lceil n/2\rceil}\end{pmatrix}=\begin{pmatrix}
	 		xI_{\lfloor n/2\rfloor} -\frac{1}{x}BB^\T &-B\\
	 		0& xI_{\lceil n/2\rceil}
	 		\end{pmatrix},
	 	\end{equation*}
	 	
	 	we obtain
	 \begin{equation*}
	 \det	\begin{pmatrix}
	 	xI_{\lfloor n/2\rfloor} &-B\\
	 	-B^\T& xI_{\lceil n/2\rceil}
	 \end{pmatrix}=\det\begin{pmatrix}
	 xI_{\lfloor n/2\rfloor}-\frac{1}{x}BB^\T &-B\\
	 0& xI_{\lceil n/2\rceil}
	 \end{pmatrix}=x^{-\lfloor n/2\rfloor}\det (x^2 I-BB^\T) x^{\lceil n/2\rceil},
	 \end{equation*}
	 i.e., $\chi(A;x)=x^\delta \chi(BB^\T;x^2)$.
	 Similarly, by the identity
	 	\begin{equation*}
	 	\begin{pmatrix}
	 		xI_{\lfloor n/2\rfloor} &-B\\
	 		-B^\T& xI_{\lceil n/2\rceil}
	 	\end{pmatrix}\begin{pmatrix}I_{\lfloor n/2\rfloor}&\frac{1}{x}B\\0 &I_{\lceil n/2\rceil}\end{pmatrix}=\begin{pmatrix}
	 		xI_{\lfloor n/2\rfloor} &0\\
	 		-B^\T& xI_{\lceil n/2\rceil} -\frac{1}{x}B^\T B
	 	\end{pmatrix},
	 	\end{equation*}
	 	we obtain $\chi(A;x)=x^{-\delta} \chi(B^\T B;x^2)$. This completes the proof.	
	 \end{proof}
	 The following basic connection between $\Delta_A$ and $\Delta_{BB^\T}$ was obtained by Ji et al.~\cite{ji2025} for the case that $n$ is even. We include the short proof for completeness.
	 	\begin{lemma}[\cite{ji2025}] \label{abb}
	  $\Delta_{A}=  4^{ \lfloor n/2 \rfloor }\Delta^2_{BB^\T}$ and hence $2^{-\lfloor n/2\rfloor} \sqrt{\Delta_A}=\Delta_{BB^\T}$. 
	 \end{lemma}
	 \begin{proof} 	
	 	 By Lemma \ref{chiab}, we have $\chi(A;x)=x^\delta\chi(BB^\T;x^2)$.  Denote $m=\lfloor n/2\rfloor$ and let the spectrum of $BB^\T$ be $\spec(BB^\T) =\{\lambda_1^2,\ldots,\lambda_m^2\}$. Then we have  
	 	 \begin{equation}\label{saeo}
	 	 	\spec (A)=\begin{cases}
	 	 		\{\lambda_1,\ldots,\lambda_m\}\cup\{-\lambda_1,\ldots,-\lambda_m\}& n \text{~even;}\\
	 	 		\{\lambda_1,\ldots,\lambda_m\}\cup\{-\lambda_1,\ldots,-\lambda_m\}\cup \{0\}&n \text{~odd.}
	 	 	\end{cases}
 	 	 	 \end{equation}
	 First consider the case that $n$ is even. Then, we have
	 	\begin{eqnarray}\label{ec}
	 		\Delta _{A}&=&\prod _{1\leq i< j\leq m}\left( \lambda _{j}-\lambda _{i}\right) ^{2}\prod _{1\leq i< j\leq m}\left( -\lambda _{j}+\lambda _{i}\right) ^{2}\prod _{\substack{1\le i\le m\\1\le  j \le m}}\left( -\lambda _{j}-\lambda _{i}\right) ^{2} \nonumber\\
	 		&=&\left(\prod _{1\leq i< j\leq m}\left( \lambda _{j}-\lambda _{i}\right) ^{2}\right)^2\left(\prod_{1\le i\le m} (2\lambda_i)^2\right)\left(\prod_{1\le i<j\le m}\left( \lambda_{j}+\lambda_i\right) \right) ^{2}   \nonumber        \\
	 		&=&4^{m}\prod_{1\le i\le m}\lambda_i^2\left(\prod_{1\le i<j\le m}\left( \lambda _{j}^{2}-\lambda _{i}^{2}\right) ^{2}\right) ^{2} \nonumber\\
	 		&=&4^{ m }\det (BB^\T) \Delta^2_{BB^\T}.
	 	\end{eqnarray}  
	 	Noting that $\det(B B^\T)=\pm \det A=\pm c_0=\pm 1$ and $\det(B B^\T)=(\det(B))^2\ge 0$, we must have $\det(B B^\T)=1$ and hence Eq.~\eqref{ec} reduces to $\Delta_A= 4^{m}\Delta^2_{BB^\T}$. This proves the lemma for the even case.
	 	
	Now we consider the case that $n$ is odd. Denote $R=4^{ m }\det (BB^\T) \Delta^2_{BB^\T}$, which is the result of $\Delta_A$ for the even case.   From Eq.~\eqref{saeo}, it is not difficult to see that for odd $n$, 
	\begin{equation}\label{oc}
		\Delta_A=R\times \left(\prod_{1\le i\le m}{\lambda_i^2}\right)^2=4^{ m }\left(\det (BB^\T)\right)^3 \Delta^2_{BB^\T}.
	\end{equation}
	 Recall that the coefficient of the linear term in $\chi(A;x)$  is $\pm 1$. Differentiating both sides of the equation $\chi(A;x)=x\cdot\chi(BB^\T;x^2)$ and evaluating at $x=0$ gives $\chi(BB^\T;0)=\pm 1$, i.e., $\det(BB^\T)=\pm 1$. As $BB^\T$ is positive semidefinite,  we must have $\det( BB)^\T \ge 0$ and hence $\det(BB^\T)= 1$. Thus, Eq.~\eqref{oc} also reduces to $\Delta_A= 4^{m}\Delta^2_{BB^\T}$. This completes the proof.
	 \end{proof} 
	 
Let $Q$ be any matrix in $\mathcal{Q}(\Sigma)$. We shall prove Theorem \ref{main} by establishing the following two propositions.

\begin{proposition}\label{p1}
	If $\ell(Q)$ is even then so is $\Delta_{BB^\T}$.
\end{proposition}

\begin{proposition}\label{p2}
If $p$ is an odd prime factor of $\ell(Q)$ then $p^2\mid \Delta_{BB^\T}$.
\end{proposition}  

The proofs of these Propositions \ref{p1} and \ref{p2} will be presented in the following two subsections. 
It turns out  that Theorem \ref{main} is an easy consequence of these two propositions.

\noindent\textbf{Proof of Theorem \ref{main}}\quad   Suppose that $2^{-\lfloor n/2\rfloor}\sqrt{\Delta_\Sigma}$ is squarefree. Let $Q$ be any matrix in $\mathcal{Q}(\Sigma)$. By Lemma \ref{abb}, $\Delta_{BB^\T}=2^{-\lfloor n/2\rfloor}\sqrt{\Delta_\Sigma}$ and hence is squarefree. In particular, $4\nmid \Delta_{BB^\T}$, i.e., $\Delta_{BB^\T}\not\equiv  0\pmod{4}$ . By Proposition \ref{not2}, $\Delta_{BB^\T}\not\equiv 2\pmod{4}$. It follows that  $\Delta_{BB^\T}\equiv 0,1\pmod{4}$, i.e., $\Delta_{BB^\T}$ is odd.  Proposition \ref{p1} implies that $\ell(Q)$ must be odd. Moreover, as $\Delta_{BB^\T}$ is squarefree, Proposition \ref{p2} implies that $\ell(Q)$ has no odd prime factor. Thus, $\ell(Q)=1$ and hence   $\Sigma$ is DGS by Lemma  \ref{dgslev}. This completes the proof of Theorem \ref{main}. \hfill \qedsymbol

\begin{remark}\normalfont{
As a byproduct of the proof of Theorem \ref{main}, we note that if $\Sigma$ is almost controllable and satisfies the conditions of Theorem \ref{main}, then $\mathcal{Q}(\Sigma)$ contains only permutation matrix and hence $\Sigma$ must have a nontrivial automorphism.}
\end{remark}

	\subsection{The case $p=2$}
The main aim of this subsection is to prove Proposition \ref{p1}. We fix $ p=2 $ here, and for simplicity, we omit the subscript $p$ for some notations. For example, $\col(M)$ means $\col_2(M)$, which is a subspace of $\mathbb{F}_2^n$.
Let $ v \in \mathbb{Z}^n $ be an integer vector and $V$ be a subspace of $\mathbb{F}_2^n$. By slight abuse of notation, we write $v \in V $ to mean that the reduction of $v$ modulo $2$ lies in $V$.
  \begin{lemma}[\cite{wang2013ElJC}]\label{m4}
Suppose $Q\in \mathcal{Q}(\Sigma)$ with even level. Then for any integer vector $q\in \col(\hat{Q})\subset\mathbb{F}_2^n$ and any nonnegative integer $k$, we have $q^\T A^{k} q\equiv 0\pmod{4}$.
\end{lemma}
In the following, we always assume $\ell(Q)$ is even. Thus, by Lemma \ref{ainv} for $p=2$, we know that $\col(\hat{Q})$ is a nonzero and totally isotropic $A$-invariant subspace. It follows from Lemma \ref{pd} that the characteristic polynomial $\chi(A;x)\in \mathbb{F}_2[x]$ has a multiple factor $\phi(x)$ such that
 $$\col(\hat{Q})\cap \ker \phi(A)\neq 0.$$
We shall show that $\phi(x)$ is also a multiple factor of $\chi(BB^\T;x)$, which clearly completes the Proposition \ref{p1} by  Lemma \ref{mf}. We first show that $\phi(x)$ is indeed a factor of $\chi(BB^\T ;x)$.

\begin{lemma}\label{rt}
	$\phi(x)\mid \chi(BB^\T;x)$ and $\phi(0)=1$ over $\mathbb{F}_2$.
\end{lemma}
\begin{proof}
	Since we are working over $\mathbb{F}_2$, we have $f(x^2)=f^2(x)$ when $f$ is a polynomial. Thus, the first equality in Lemma \ref{chiab} becomes 
	$\chi(A;x)=x^\delta (\chi(BB^\T;x))^2$.  Noting that $\delta\le 1$ and $\phi(x)$ is a multiple factor of $\chi(A;x)$, we must have $\phi(x)\mid \chi(BB^\T;x)$. It remains to show that $\phi(0)=1$. Suppose to the contrary that $\phi(0)=0$. Since $\phi(x)$ is a multiple factor of $\chi(A;x)$ (over $\mathbb{F}_2$), we see that the coefficients of the constant term and the linear term in $\chi(A;x)$ are both 0 over $\mathbb{F}_2$, i.e., both are even as ordinary integers. But this contradicts the requirement that  $c_\delta=\pm 1$. Thus, $\phi(0)=1$ and the proof is complete.
	\end{proof}

\begin{lemma}\label{bv}
	There exist two vectors $u\in \mathbb{F}_2^{\lfloor n/2\rfloor}$ and $v\in \mathbb{F}_2^{\lceil n/2\rceil}$ such that $Bv\neq 0$ and $\begin{pmatrix} u\\v\end{pmatrix}\in \col(\hat{Q})\cap\ker \phi(A)$.	
\end{lemma}
\begin{proof}
	Let $q=(q_1,\ldots,q_n)^\T\in \mathbb{F}_2^n$ be any nonzero vector in $\col(\hat{Q})\cap \ker \phi(A).$  Denote $u=(q_1,q_2,\ldots,q_{\lfloor n/2\rfloor})^\T$ and $v=(q_{\lfloor n/2\rfloor+1},\ldots,q_{n-1},q_{n})^\T$.

	\noindent\textbf{Claim 1}: $Aq\in \col(\hat{Q})\cap \ker \phi(A)$ and $Aq \neq0$.
	
	The first assertion should be clear as both $\col(\hat{Q})$ and $\ker \phi(A)$ are $A$-invariant. It remains to show $Aq\neq 0$. Suppose to the contrary that $Aq=0$. Then we  have  $\phi(A)q=\phi(0)q$. By Lemma \ref{rt}, we have $\phi(0)q=q$ and hence is nonzero. But clearly, $\phi(A)q=0$. This is a contradiction and hence Claim 1 follows.

	\noindent\textbf{Claim 2}: If $u=0$ then $Bv\neq 0$.
		 
	 As $	 	A=\begin{pmatrix} 
	 		0   & B \\
	 		B^\T  & 0 
	 	\end{pmatrix}$ and $q=\begin{pmatrix}u\\v\end{pmatrix}=\begin{pmatrix}0\\v\end{pmatrix}$, we see that $Aq=\begin{pmatrix}Bv\\ 0\end{pmatrix}$. Thus, by Claim 1, $Bv\neq 0$ and hence Claim 2 follows.

By Claim 2, to complete the proof of Lemma \ref{bv}, it suffices to consider the case that $u\neq 0$ and $Bv=0$ hold simultaneously. Now let $\tilde{q}=Aq$ and let $\tilde{u}\in \mathbb{F}_2^{\lfloor n/2\rfloor}$ and  $\tilde{v}\in \mathbb{F}_2^{\lceil n/2\rceil}$ such that $\tilde{q}=\begin{pmatrix} \tilde{u}\\\tilde{v}\end{pmatrix}$. Then we have
\begin{equation*}
	\begin{pmatrix} \tilde{u}\\ \tilde{v}\end{pmatrix}=\begin{pmatrix} 
		0   & B \\
		B^\T  & 0 
	\end{pmatrix}\begin{pmatrix}u\\v\end{pmatrix}=\begin{pmatrix}Bv\\B^\T u\end{pmatrix}=\begin{pmatrix}0\\B^\T u\end{pmatrix}.
\end{equation*}
Now we show that $\tilde{u}$ (which is 0) and $\tilde{v}$ satisfy all requirements of Lemma \ref{bv}. By Claim 1, we know that $\tilde{q}$ is a nonzero vector in $\col(\hat{Q})\cap \ker \phi(A)$. Consequently, using Claim 2 for the vector $\tilde{q}=\begin{pmatrix} \tilde{u}\\\tilde{v}\end{pmatrix}=\begin{pmatrix} 0\\\tilde{v}\end{pmatrix}$, we obtain that $B\tilde{v}\neq 0$. This completes the proof of Lemma \ref{bv}.
\end{proof}
\begin{lemma}\label{mf}
 $\phi^2(x)\mid\chi(BB^\T;x)$ over $\mathbb{F}_2$.
\end{lemma}
\begin{proof}
	Suppose to the contrary that $\phi^2(x)\nmid\chi(BB^\T;x)$. Then Lemma \ref{rt} implies that $\phi(x)$ is a simple factor of $\chi(BB^\T;x)$. According to Lemma \ref{bv}, there exists a vector $q=\begin{pmatrix} u\\v\end{pmatrix}\in \mathbb{Z}^n$ such that $Bv\not\equiv 0\pmod{2}$ and the reduction of $q$ over $\mathbb{F}_2$ lies in $\col(\hat{Q})\cap \ker \phi(A)$.
	
	\noindent\textbf{Claim 1}: $u,Bv\in \ker \phi(BB^\T)$, i.e., $\phi(BB^\T) u\equiv \phi(BB^\T)Bv\equiv 0\pmod{2} $.
	
	Since $f(x^2)\equiv f^2(x)\pmod{2}$ for any $f\in \mathbb{Z}[x]$, we have
	$$\phi(A^2)\begin{pmatrix}u\\v\end{pmatrix}\equiv \phi(A)\phi(A)\begin{pmatrix}u\\v\end{pmatrix}\equiv 0\pmod{2}.$$
	On the other hand, noting that $A^2=\begin{pmatrix} 
		0   & B \\
		B^\T  & 0 
	\end{pmatrix}^2=\begin{pmatrix} 
	BB^\T   & 0 \\
	 0 & B^\T B 
	\end{pmatrix}$, we obtain 
	
		$$\phi(A^2)\begin{pmatrix}u\\v\end{pmatrix}\equiv \begin{pmatrix} 
			\phi(BB^\T)   & 0 \\
			0 & \phi(B^\T B) 
		\end{pmatrix} \begin{pmatrix}u\\v\end{pmatrix}\equiv\begin{pmatrix}\phi(BB^\T) u\\ \phi(B^\T B)v\end{pmatrix} \pmod{2}.$$
	It follows that $\phi(BB^\T) u\equiv 0\pmod{2}$. 
	
	 Note that $\begin{pmatrix} Bv\\B^\T u\end{pmatrix}=\begin{pmatrix}0&B\\B^\T&0\end{pmatrix}\begin{pmatrix}u\\v\end{pmatrix}=Aq\in \ker \phi(A)$. The same argument indicates that $\phi(BB^\T)(Bv)\equiv 0\pmod{2}$.  Claim 1 follows.
	 
	 Since the irreducible polynomial  $\phi(x)$ is a simple  factor of $\chi(BB^T;x)$, we see that $\dim \ker \phi(A)=\deg \phi(x)$ and  $\ker \phi(A)$ is a cyclic subspace generated by any nonzero vector in it. Let $d=\deg \phi(x)$. Note that $Bv\neq 0$. Then it follows from  Claim 1 that (over $\mathbb{F}_2$)
	 \begin{equation}\label{sp1}
	 	\ker \phi(BB^\T)=\span\{Bv,(BB^\T)Bv,\ldots,(BB^\T)^{d-1}Bv\}
	 \end{equation}
	 and 
	 \begin{equation}\label{sp2}
	 	\ker \phi(BB^\T)=\span\{u,(BB^\T)u,\ldots,(BB^\T)^{d-1}u\} \text{~when~} u\neq 0.
	 \end{equation}
	 \noindent\textbf{Claim 2}: $\ker \phi(BB^\T)$ is totally isotropic.
	 
	 We first consider the case that $u\equiv 0\pmod{2}$. As $\begin{pmatrix} 0\\v\end{pmatrix}\in \col(\hat{Q})$,  Lemma \ref{m4} implies that, for any $k\ge 0$,
	 \begin{equation}\label{vbv}
	 	(0^\T,v^\T)\begin{pmatrix}0&B\\B^\T&0\end{pmatrix}^{2k}\begin{pmatrix}0\\v\end{pmatrix}\equiv 0\pmod 4, \text{~i.e.,~}  v^\T (B^\T B)^k v\equiv 0\pmod{4}.
	 \end{equation}
	   Note that $Bv, (BB^\T)Bv,\ldots,(BB^\T)$ constitute a basis of $\ker \phi(BB^\T)$. To show that $\ker \phi(BB^\T)$ is totally isotropic, it suffices to show that any two vectors $\alpha$ and $\beta$ in the basis, say $\alpha=(BB^\T)^i Bv$ and  $\beta=(BB^\T)^j Bv$,  are orthogonal over $\mathbb{F}_2$. Direct computation shows that  
	   $$\alpha^\T \beta=((BB^\T)^iBv)^\T(BB^\T)^j Bv)=v^\T(B^\T B)^{i+j+1}v.$$
	   Thus, by Eq.~\eqref{vbv}, we have $\alpha^\T \beta\equiv 0\pmod{4}$, which clearly implies  $\alpha^\T \beta\equiv 0\pmod{2}$, or equivalently, $\alpha \perp \beta$ over $\mathbb{F}_2$. This proves Claim 2 for the case that $u\equiv 0\pmod{2}$.
	   
	   Now assume $u\not\equiv 0\pmod{2}$. By Eqs.~\eqref{sp1} and \eqref{sp2}, it suffice to show that 
	   \begin{equation}\label{buv}
	 (BB^\T)^i u\perp  (BB^\T)^j Bv,  \text{~over~} \mathbb{F}_2,\text{~for any~} i,j\ge 0.
	\end{equation}
	    As $\begin{pmatrix}u\\v\end{pmatrix}\in \col(\hat{Q})$, Lemma \ref{m4} implies that, for any $k\ge 0$,
	     \begin{equation*}
	    	(u^\T,v^\T)\begin{pmatrix}0&B\\B^\T&0\end{pmatrix}^{2k+1}\begin{pmatrix}u\\v\end{pmatrix}=(u^\T,v^\T)\begin{pmatrix}(BB^\T)^k&0\\0&(B^\T B)^k\end{pmatrix}\begin{pmatrix}Bv\\B^\T u\end{pmatrix}\equiv 0\pmod{4},
	    \end{equation*}
	    i.e., $u^\T (BB^\T)^k Bv+v^\T (B^\T B)^k B^\T u\equiv 0\pmod{4}$. Since the two additive terms on the left-hand side are equal (as can be seen by taking the transpose of one term), we obtain $u^\T (BB^\T)^k Bv\equiv 0\pmod{2}$. Taking $k=i+j$, it follows that
	       $$((BB^\T)^i u)^\T 	(BB^\T)^j Bv =u^\T(BB^\T)^{i+j}Bv\equiv 0\pmod{2},$$
	    that is, Eq.~\eqref{buv} holds. This completes the proof of Claim 2.
	    
	    Let $M=BB^\T$ and $U=\ker \phi(M)$. Clearly $U$ is $M$-invariant. By Claim 2, $U$ is totally isotropic.  As $\phi(x)$ is a simple factor of $\chi(M;x)$, we find that $U\cap \ker \phi(M)=0$ by Lemma \ref{pd}. This is a contradiction as $U=\ker \phi(M)$ and is nonzero. The proof of Lemma \ref{mf} is complete.
\end{proof}
\noindent\textbf{Proof of Proposition \ref{p1}} Let $f= \chi(BB^\T;x)\in \mathbb{F}_2[x]$. We know from Lemma \ref{mf}, that $\phi(x)$ is a multiple factor of $f(x)$ over $\mathbb{F}_2$. Thus, Lemma \ref{mfbasic} implies that $2\mid \Delta(f)$, completing the proof. \hfill\qedsymbol

	\subsection{The case $p$ is odd}
	The main aim of this subsection is to prove Proposition \ref{p2} by contradiction.
	
	\noindent\textbf{Proof of Proposition \ref{p2}}
	 Suppose to the contrary that there exists an odd prime $p$ such that $p\mid\ell(Q)$ but $p^2\nmid \Delta_{BB^\T}$. Then by Lemma \ref{dd}, we have $p\mid\Delta _{A}$.  On the other hand, we know from Lemma \ref{abb} that $\Delta_A=4^{\lfloor n/2\rfloor}\Delta^2_{BB^\T}$. Thus, $p\mid \Delta_{BB^\T}$ as $p$ is an odd prime.  It follows from Lemma \ref{d} that there exists an integer $\lambda_0$ and a polynomial $\phi(x)$ such that $\chi(BB^\T;x)=(x-\lambda_0)^2\varphi(x)$ over $\mathbb{F}_p$. Noting that $\chi(A;x)=x^\delta \chi(BB^\T;x^2)$ by Lemma \ref{chiab}, we obtain $\chi(A;x)=x^\delta (x^2-\lambda_0) ^2 \varphi(x^2)$.  Denote $\psi(x)=x^\delta\varphi(x^2)$. Then we can write 
	\begin{equation}\label{af}
		\chi(A;x)=(x^2-\lambda_0)^2 \psi(x) \text{~over~} \mathbb{F}_p.
		\end{equation}
	
Since the constant term or linear coefficient of $\chi(A;x)$ is $\pm 1$ (and hence nonzero modulo $p$), we find from Eq.~\eqref{af} that $\lambda_0\not\equiv 0\pmod{p}$. Let $S$ be the Sylvester matrix of $\chi(A;x)$ and its derivative  $\chi'(A;x)$. Noting that $p^2\nmid \Delta_{BB^\T}$ and $\Delta_A=4^{\lfloor n/2\rfloor}\Delta^2_{BB^\T}$, we obtain $p^3\nmid \Delta_A$, or equivalently, $p^3\nmid \det S$.  It follows from Lemma \ref{snf} that  $\dim\ker S\le 2$ over $\mathbb{F}_p$. Consequently, by Lemma \ref{degker}, we obtain \begin{equation*}
\deg \gcd(\chi(A;x),\chi'(A;x))=\dim\ker S\le 2.
\end{equation*}
Therefore, we see from Eq.~\eqref{af} that $\psi(x)$ is squarefree and coprime to $(x^2-\lambda_0)$, since otherwise the polynomial $\gcd(\chi(A;x),\chi'(A;x)$ would have degree at least 3, a contradiction.

\noindent\textbf{Claim}:  $p^2\mid \det(\lambda_0 I-BB^\T)$.

We prove the Claim by considering  two cases:

\noindent\emph{Case 1}: $x^2-\lambda_0$ is irreducible over $\mathbb{F}_p$.

As $A^2=\diag(BB^\T,B^\T B)$, we find that $\chi(A^2;x)=\chi(BB^\T;x)\chi(B^\T B;x)$ and hence
\begin{equation}\label{a2b}
\chi(A^2;\lambda_0)=\chi(BB^\T;\lambda_0)\chi(B^\T B;\lambda_0).
\end{equation} By Lemma \ref{chiab}, we have $\chi(B^\T B;x^2)=(x^2)^\delta \chi(BB^\T;x^2)$, i.e., $\chi(B^\T B;x)=x^\delta \chi(BB^\T;x)$. Taking $x=\lambda_0$, we obtain $\chi(B^\T B;\lambda_0)=\lambda_0^\delta \chi(BB^\T;\lambda_0)$. Noting that $\lambda_0\neq 0\pmod{p}$, we find that, for any $k\ge 1$, 
\begin{equation}\label{pkb}
p^k\mid \chi(BB^\T;\lambda_0) \text{~if and only if~} p^k\mid \chi(B^\T B;\lambda_0).
\end{equation}
On the other hand, since $\psi(x)$ is squarefree coprime to $x^2-\lambda_0$, we see that $\phi(x)=(x^2-\lambda_0)$ is the only multiple factor of $\chi(A;x)$. It follows from Proposition \ref{strict} that $p^3\mid \det\phi(A)$, i.e., $p^3\mid \chi(A^2;\lambda_0)$. This, combining with Eq.~\eqref{a2b} and Eq.~\eqref{pkb} for $k=2$, leads to $p^2\mid\chi(BB^\T;\lambda_0)$.

\noindent\emph{Case 2}: $x^2-\lambda_0$ is reducible over $\mathbb{F}_p$, i.e., $x^2-\lambda_0\equiv (x-\lambda_1)(x+\lambda_1)\pmod{p}$ for some $\lambda_1\in \mathbb{Z}$.

According to Remark \ref{nid}, the truth of the Claim does not dependent on the choice of $\lambda_0$ in its residue class modulo $p$. Consequently, since $\lambda_0\equiv \lambda_1^2\pmod{p}$, we may safely assume $\lambda_0=\lambda_1^2$. It follows from Lemma \ref{chiab} that 
\begin{equation}\label{abl}
	\chi(A;\lambda_1)=\lambda_1^\delta \chi(BB^\T;\lambda_0) \text{~and~} 	\chi(A;-\lambda_1)=(-\lambda_1)^\delta \chi(BB^\T;\lambda_0).
\end{equation}
As $\lambda_0\not\equiv 0\pmod {p}$, we see that $\lambda_1\not\equiv 0\pmod{p}$.  Note that $\chi(A;x)$ has exactly two multiple factors $\phi_1(x)=x-\lambda_1$ and $\phi_2(x)=x+\lambda_1$. It follows from Proposition \ref{strict} that $p^2\mid \det\phi_1(A)$ or $p^2\mid \det\phi_2(A)$. Nevertheless, either implies $p^2\mid \chi(BB^\T;\lambda_0)$  according to Eq.~\eqref{abl}. 

This proves the Claim. 	By Lemma \ref{d2}, the equation
\begin{eqnarray*}
	\chi(BB^\T;x)u(x)\equiv \chi(BB^\T;x)' v(x)\pmod{p^2}
\end{eqnarray*}
has a solution $(u(x),v(x))\in(\mathbb{Z}[x])^2$ satisfying:
$$u(x),v(x)\not\equiv0\pmod p,\text{~}\deg(u(x))<\lfloor n/2\rfloor-1, \text{~and~}  \deg(v) <\lfloor n/2\rfloor.$$
Let $S$ be the Sylvester matrix of $\chi(BB^\T;x)$ and its derivative $\chi'(BB^\T;x)$. Using a similar argument as in the proof of Lemma  \ref{degker}, the existence of a solution $(u,v)$ described above means that the linear equations 
$
S^\T\eta\equiv0\pmod{p^2}$ has a solution  $\eta\not\equiv0\pmod{p}$.

On the other hand, as $p^2\nmid \Delta_{BB^\T}$ and $\Delta_{BB^\T}=\pm \det S^\T$, we have  $p^2\nmid \det S^\T$.  Consequently, $p^2\nmid d_n$, where $d_n$ denotes the last invariant of $S^\T$. This contradicts Lemma \ref{snf} (\rmnum{4}) and hence completes the proof of Proposition \ref{p2}.

\section{Examples}
In this section, we present some examples to illustrate Theorem \ref{main}. All  computations were performed using Wolfram Mathematica.

\noindent\textbf{Example 1}: Let $n=13$ and $\Sigma$ be the  signed bipartite graph with adjacency matrix as follows:
\begin{equation*}
	A=\begin{tiny}\left(
		\begin{array}{ccccccccccccc}
			0 & 0 & 0 & 0 & 0 & 0 & -1 & -1 & -1 & -1 & 0 & 0 & -1 \\
			0 & 0 & 0 & 0 & 0 & 0 & 1 & 0 & 0 & 1 & 0 & 0 & 0 \\
			0 & 0 & 0 & 0 & 0 & 0 & -1 & -1 & 0 & 1 & 0 & -1 & -1 \\
			0 & 0 & 0 & 0 & 0 & 0 & -1 & 1 & 0 & 0 & 0 & -1 & -1 \\
			0 & 0 & 0 & 0 & 0 & 0 & 1 & 1 & 0 & 1 & 0 & 0 & 0 \\
			0 & 0 & 0 & 0 & 0 & 0 & 0 & -1 & 1 & -1 & 0 & -1 & -1 \\
			-1 & 1 & -1 & -1 & 1 & 0 & 0 & 0 & 0 & 0 & 0 & 0 & 0 \\
			-1 & 0 & -1 & 1 & 1 & -1 & 0 & 0 & 0 & 0 & 0 & 0 & 0 \\
			-1 & 0 & 0 & 0 & 0 & 1 & 0 & 0 & 0 & 0 & 0 & 0 & 0 \\
			-1 & 1 & 1 & 0 & 1 & -1 & 0 & 0 & 0 & 0 & 0 & 0 & 0 \\
			0 & 0 & 0 & 0 & 0 & 0 & 0 & 0 & 0 & 0 & 0 & 0 & 0 \\
			0 & 0 & -1 & -1 & 0 & -1 & 0 & 0 & 0 & 0 & 0 & 0 & 0 \\
			-1 & 0 & -1 & -1 & 0 & -1 & 0 & 0 & 0 & 0 & 0 & 0 & 0 \\
		\end{array}
		\right).\end{tiny}
\end{equation*}
We obtain the factorization of $\chi(A;x)$ over $\mathbb{Q}$:
$$-x \left(x^{12}-24 x^{10}+194 x^8-679 x^6+1022 x^4-496 x^2+1\right),$$ 
along with $$\det W(\Sigma)=2^{6}\times11^3\times3413\times697913,$$
and $$2^{-n/2}\sqrt{\Delta_{\Sigma}}=107\times15259\times12978894869.$$
By Theorem~\ref{main}, we conclude that $\Sigma$ is $\mathrm{DGS}$.

    \noindent\textbf{Example 2}: Let $n=14$ and $\Sigma$ be the controllable signed bipartite graph with adjacency matrix as follows:
\begin{equation*}
	A=\begin{tiny}\left(
		\begin{array}{cccccccccccccc}
		0 & 0 & 0 & 0 & 0 & 0 & 0 & 1 & 1 & 0 & 0 & -1 & 0 & 0 \\
		0 & 0 & 0 & 0 & 0 & 0 & 0 & 0 & 1 & 0 & 0 & 0 & 0 & 0 \\
		0 & 0 & 0 & 0 & 0 & 0 & 0 & 0 & 0 & 1 & 0 & 1 & 0 & 0 \\
		0 & 0 & 0 & 0 & 0 & 0 & 0 & 0 & 0 & 0 & 1 & 0 & -1 & 0 \\
		0 & 0 & 0 & 0 & 0 & 0 & 0 & 0 & 0 & 0 & 0 & 1 & 1 & 1 \\
		0 & 0 & 0 & 0 & 0 & 0 & 0 & 0 & 0 & 0 & 0 & 0 & 1 & 0 \\
		0 & 0 & 0 & 0 & 0 & 0 & 0 & 0 & 0 & 0 & 0 & 0 & 0 & 1 \\
		1 & 0 & 0 & 0 & 0 & 0 & 0 & 0 & 0 & 0 & 0 & 0 & 0 & 0 \\
		1 & 1 & 0 & 0 & 0 & 0 & 0 & 0 & 0 & 0 & 0 & 0 & 0 & 0 \\
		0 & 0 & 1 & 0 & 0 & 0 & 0 & 0 & 0 & 0 & 0 & 0 & 0 & 0 \\
		0 & 0 & 0 & 1 & 0 & 0 & 0 & 0 & 0 & 0 & 0 & 0 & 0 & 0 \\
		-1 & 0 & 1 & 0 & 1 & 0 & 0 & 0 & 0 & 0 & 0 & 0 & 0 & 0 \\
		0 & 0 & 0 & -1 & 1 & 1 & 0 & 0 & 0 & 0 & 0 & 0 & 0 & 0 \\
		0 & 0 & 0 & 0 & 1 & 0 & 1 & 0 & 0 & 0 & 0 & 0 & 0 & 0 \\
	\end{array}
\right).\end{tiny}
\end{equation*}
The characteristic polynomial $\chi(A;x)$ factors over $\mathbb{Q}$ as:
$$\left(x^7-x^6-6 x^5+4 x^4+9 x^3-4 x^2-3 x+1\right) \left(x^7+x^6-6 x^5-4 x^4+9 x^3+4 x^2-3 x-1\right),$$ 
with $\det W(\Sigma)=2^{14}$ and $2^{-n/2}\sqrt{\Delta_{\Sigma}}=17\times23\times64879$.
We observe that this case does not satisfy the conditions of Theorem~\ref{main2}. However, by Theorem~\ref{main}, we conclude that $\Sigma$ is $\mathrm{DGS}$.

 \noindent\textbf{Example 3}: Let $n=14$ and $\Sigma$ be the almost controllable signed bipartite graph with adjacency matrix as follows:
\begin{equation*}
	A=\begin{tiny}\left(
	\begin{array}{cccccccccccccc}
		0 & 0 & 0 & 0 & 0 & 0 & 0 & 1 & 0 & 0 & -1 & 0 & -1 & -1 \\
		0 & 0 & 0 & 0 & 0 & 0 & 0 & 0 & 1 & 0 & -1 & 0 & -1 & 0 \\
		0 & 0 & 0 & 0 & 0 & 0 & 0 & 0 & 0 & 1 & 0 & 0 & 0 & 0 \\
		0 & 0 & 0 & 0 & 0 & 0 & 0 & 0 & 0 & 0 & 1 & -1 & -1 & 0 \\
		0 & 0 & 0 & 0 & 0 & 0 & 0 & 0 & 0 & 0 & 0 & 1 & 0 & -1 \\
		0 & 0 & 0 & 0 & 0 & 0 & 0 & 0 & 0 & 0 & 0 & 0 & 1 & 1 \\
		0 & 0 & 0 & 0 & 0 & 0 & 0 & 0 & 0 & 0 & 0 & 0 & 0 & 1 \\
		1 & 0 & 0 & 0 & 0 & 0 & 0 & 0 & 0 & 0 & 0 & 0 & 0 & 0 \\
		0 & 1 & 0 & 0 & 0 & 0 & 0 & 0 & 0 & 0 & 0 & 0 & 0 & 0 \\
		0 & 0 & 1 & 0 & 0 & 0 & 0 & 0 & 0 & 0 & 0 & 0 & 0 & 0 \\
		-1 & -1 & 0 & 1 & 0 & 0 & 0 & 0 & 0 & 0 & 0 & 0 & 0 & 0 \\
		0 & 0 & 0 & -1 & 1 & 0 & 0 & 0 & 0 & 0 & 0 & 0 & 0 & 0 \\
		-1 & -1 & 0 & -1 & 0 & 1 & 0 & 0 & 0 & 0 & 0 & 0 & 0 & 0 \\
		-1 & 0 & 0 & 0 & -1 & 1 & 1 & 0 & 0 & 0 & 0 & 0 & 0 & 0 \\
	\end{array}
	\right).\end{tiny}
\end{equation*}
The characteristic polynomial $\chi(A;x)$ factors over $\mathbb{Q}$ as:
$$(x-1) (x+1) \left(x^{12}-15 x^{10}+75 x^8-151 x^6+111 x^4-23 x^2+1\right),$$ 
with $\rank~W(\Sigma)=13$ and $2^{-n/2}\sqrt{\Delta_{\Sigma}}=13\times45953\times106501$.
By Theorem~\ref{main}, we conclude that $\Sigma$ is $\mathrm{DGS}$.

	\section*{Acknowledgments}
	This work was partially supported by the National Natural Science Foundation of China (Grant No.~12001006) and Wuhu Science and Technology Project, China (Grant No.~2024kj015). 	
	
\end{document}